\documentclass{amsart}
\usepackage{mathrsfs,latexsym,amsfonts,amssymb}

\newtheorem{theorem}{Theorem}[section]
\newtheorem{lemma}[theorem]{Lemma}

\newtheorem{question}[theorem]{Question}
\theoremstyle{definition}
\newtheorem{definition}[theorem]{Definition}

\theoremstyle{remark}

\begin{document}

\title[topological monomorphisms between free paratopological groups]
{topological monomorphisms between free paratopological groups}

\author{Fucai Lin}
\address{Fucai Lin(corresponding author): Department of Mathematics and Information Science,
Zhangzhou Normal University, Zhangzhou 363000, P. R. China}
\email{linfucai2008@yahoo.com.cn}

\thanks{Supported by the NSFC (No. 10971185, No. 10971186) and the Natural Science Foundation of Fujian Province (No. 2011J05013) of China.}


\begin{abstract}
Suppose that $X$ is a subspace of a Tychonoff space $Y$. Then the embedding mapping $e_{X, Y}: X\rightarrow Y$ can be extended to a continuous monomorphism $\hat{e}_{X, Y}: AP(X)\rightarrow AP(Y)$, where $AP(X)$ and $AP(Y)$ are the free Abelian paratopological groups over $X$ and $Y$, respectively. In this paper, we mainly discuss when $\hat{e}_{X, Y}$ is a topological monomorphism, that is, when $\hat{e}_{X, Y}$ is a topological embedding of $AP(X)$ to $AP(Y)$.
\end{abstract}

\maketitle

\section{Introduction}
In 1941, free topological groups were introduced by A.A. Markov in \cite{MA} with the clear idea of extending the well-known construction of a free group from group theory to topological groups. Now, free topological groups have become a powerful tool of study in the theory of topological groups and serve as a source of various examples and as an instrument for proving new theorems, see \cite{A2008, GM, NP}.

In \cite{GM}, M.I. Graev extended continuous pseudometrics on a space $X$ to invariant continuous pseudometrics on $F(X)$ (or $A(X)$). Apparently, the description of a local base at the neutral element of the free Abelian topological group $A(X)$ in terms of continuous pseudometric on $X$ was known to M.I. Graev, but appeared explicitly in \cite{MD} and \cite{NP}. When working with free topological groups, it is also very important to know under which conditions on a subspace $X$ of a Tychonoff space $Y$, the subgroup $F(X, Y)$ of $F(Y)$ generated by $X$ is topologically isomorphic to the group $F(X)$, under the natural isomorphism extending the identity embedding of $X$ to $Y$. V.G. Pestov and E.C. Nummela gave some answers (see e.g. Theorem~\ref{t1}) in \cite{PV} and \cite{NE}, respectively. In the Abelian case, M.G. Tkackenko gave an answer in \cite{TM1}, see Theorem~\ref{t2}.

It is well known that paratopological groups are
good generalizations of topological groups, see e.g. \cite{A2008}. The Sorgenfrey line
(\cite[Example 1.2.2]{E1989}) with the usual addition is a
first-countable paratopological group but not a topological group. The absence of continuity of inversion,
the typical situation in paratopological groups, makes the study in this area
very different from that in topological groups. Paratopological groups attract
a growing attention of many mathematicians and articles in recent years.
As in free topological groups, S. Romaguera, M. Sanchis and M.G. Tkackenko in \cite{RS} define free paratopological groups. Recently, N.M. Pyrch has investigated some properties of free paratopological groups, see \cite{PN1, PN, PN2}. In this paper, we will discuss the topological monomorphisms between free paratopological groups, and extend several results valid
for free (abelian) topological groups to free (abelian) paratopological groups.
\maketitle

\section{Preliminaries}
Firstly, we introduce some notions and terminology.

Recall that a {\it topological group} $G$ is a group $G$ with a
(Hausdorff) topology such that the product mapping of $G \times G$ into
$G$ is jointly continuous and the inverse mapping of $G$ onto itself
associating $x^{-1}$ with an arbitrary $x\in G$ is continuous. A {\it
paratopological group} $G$ is a group $G$ with a topology such that
the product mapping of $G \times G$ into $G$ is jointly continuous.

\begin{definition}\cite{MA}
Let $X$ be a subspace of a topological group $G$. Assume that
\begin{enumerate}
\item The set $X$ generates $G$ algebraically, that is $<X>=G$;

\item  Each continuous mapping $f: X\rightarrow H$ to a topological group $H$ extends to a continuous homomorphism $\hat{f}: G\rightarrow H$.
\end{enumerate}
Then $G$ is called the {\it Markov free topological group on} $X$ and is denoted by $F(X)$.
\end{definition}

\begin{definition}\cite{RS}
Let $X$ be a subspace of a paratopological group $G$. Assume that
\begin{enumerate}
\item The set $X$ generates $G$ algebraically, that is $<X>=G$;

\item  Each continuous mapping $f: X\rightarrow H$ to a paratopological group $H$ extends to a continuous homomorphism $\hat{f}: G\rightarrow H$.
\end{enumerate}
Then $G$ is called the {\it Markov free paratopological group on} $X$ and is denoted by $FP(X)$.
\end{definition}

Again, if all the groups in the above definitions are Abelian, then we get the definitions of the {\it Markov free Abelian topological group} and the {\it Markov free Abelian paratopological group on} $X$ which will be denoted by $A(X)$ and $AP(X)$ respectively.

By a {\it quasi-uniform space} $(X, \mathscr{U})$ we mean the natural analog of a {\it uniform space} obtained by dropping the symmetry axiom. For each quasi-uniformity $\mathscr{U}$ the filter  $\mathscr{U}^{-1}$ consisting of the inverse relations $U^{-1}=\{(y, x): (x, y)\in U\}$ where $U\in\mathscr{U}$ is called the {\it conjugate quasi-uniformity} of $\mathscr{U}$. We recall that the standard base of the {\it left quasi-uniformity}  $\mathscr{G}_{G}$ on a paratopological group $G$ consists of the sets $$W_{U}^{l}=\{(x, y)\in G\times G: x^{-1}y\in U\},$$where $U$ is an arbitrary open neighborhood of the neutral element in $G$.
If $X$ is a subspace of $G$, then the base of the left induced quasi-uniformity $\mathscr{G}_{X}=\mathscr{G}_{G}\mid X$ on $X$ consists of the sets $$W_{U}^{l}\cap (X\times X)=\{(x, y)\in X\times X: x^{-1}y\in U\}.$$Similarly, we can define the {\it right induced quasi-uniformity} on $X$.

We also recall that the {\it universal quasi-uniformity} $\mathscr{U}_{X}$ of a space $X$ is the finest quasi-uniformity on $X$ that induces on $X$ its original topology. Throughout this paper, if $\mathscr{U}$ is a quasi-uniformity of a space $X$ then $\mathscr{U}^{\ast}$ denotes the
smallest uniformity on $X$ that contains $\mathscr{U}$, and $\tau (\mathscr{U})$ denotes the topology of $X$ generated by $\mathscr{U}$. A quasi-uniform space $(X, \mathscr{U})$ is called {\it bicomplete} if $(X, \mathscr{U}^{\ast})$ is complete.

\begin{definition}
A function $f: (X, \mathscr{U})\rightarrow (Y, \mathscr{V})$ is called {\it quasi-uniformly continuous} if for each $V\in \mathscr{V}$ there exists an $U\in\mathscr{U}$ such that $(f(x), f(y))\in V$ whenever $(x, y)\in U,$ where $\mathscr{U}$ and $\mathscr{V}$ are quasi-uniformities for $X$ and $Y$ respectively.
\end{definition}

\begin{definition}
A {\it quasi-pseudometric} $d$ on a set $X$ is a function
from $X\times X$ into the set of non-negative real numbers such that for
$x, y, z\in X$: (a) $d(x, x)=0$ and (b) $d(x, y)\leq d(x, z)+d(z, y)$. If $d$ satisfies the additional condition (c) $d(x, y)=0\Leftrightarrow x=y$, then $d$ is called a {\it quasi-metric} on $X$.
\end{definition}

Every quasi-pseudometric $d$ on $X$ generates a topology $\mathscr{F}(d)$ on $X$ which has as a
base the family of $d$-balls $\{B_{d}(x, r): x\in X, r>0\}$, where $B_{d}(x, r)=\{y\in X:
d(x, y)< r\}$.

A topological space $(X, \mathscr{F})$ is called {\it quasi-(pseudo)metrizable} if there is a quasi-
(pseudo)metric $d$ on $X$ compatible with $\mathscr{F}$, where $d$ is compatible with $\mathscr{F}$ provided $\mathscr{F}=\mathscr{F}(d).$

Denote by $\mathscr{U}^{\star}$ the upper quasi-uniformity on $\mathbb{R}$ the standard base of which consists of the sets $$U_{r}=\{(x, y)\in \mathbb{R}\times \mathbb{R}: y<x+r\},$$where $r$ is an arbitrary positive real number.

\begin{definition}
Given a group $G$ with the neutral element $e$, a function $N: G\rightarrow [0,\infty)$
is called a {\it quasi-prenorm} on $G$ if the following conditions are satisfied:
\begin{enumerate}
\item $N(e)=0$; and
\item $N(gh)\leq N(g)+N(h)$ for all $g, h\in G$.
\end{enumerate}
\end{definition}

\begin{definition}
Let $X$ be a subspace of a Tychonoff space $Y$.
\begin{enumerate}
\item The subspace $X$ is {\it P-embedded} in $Y$ if each continuous pseudometric on $X$
admits a continuous extension over $Y$;

\item The subspace $X$ is {\it P$^{\ast}$-embedded} in $Y$ if each bounded continuous pseudometric on $X$
admits a continuous extension over $Y$;

\item The subspace $X$ is {\it quasi-P-embedded} in $Y$ if each continuous quasi-pseudometric from $(X\times X, \mathscr{U}_{X}^{-1}\times \mathscr{U}_{X})$ to $(\mathbb{R}, \mathscr{U}^{\star})$
admits a continuous extension from $(Y\times Y, \mathscr{U}_{Y}^{-1}\times \mathscr{U}_{Y})$ to $(\mathbb{R}, \mathscr{U}^{\star})$;

\item The subspace $X$ is {\it quasi-P$^{\ast}$-embedded} in $Y$ if each bounded continuous quasi-pseudometric from $(X\times X, \mathscr{U}_{X}^{-1}\times \mathscr{U}_{X})$ to $(\mathbb{R}, \mathscr{U}^{\star})$
admits a continuous extension from $(Y\times Y, \mathscr{U}_{Y}^{-1}\times \mathscr{U}_{Y})$ to $(\mathbb{R}, \mathscr{U}^{\star})$.
\end{enumerate}
\end{definition}

Throughout this paper, we use $G(X)$ to denote the topological groups $F(X)$ or $A(X)$, and $PG(X)$ to denote the paratopological groups $FP(X)$ or $AP(X)$. For a subset $Y$ of a space $X$, we use $G(Y, X)$ and $PG(Y, X)$ to denote the subgroups of $G(X)$ and $PG(X)$ generated by $Y$ respectively. Moreover, we denote the abstract groups of $F(X), FP(X)$ by $F_{a}(X)$ and of $A(X)$ and $AP(X)$ by $A_{a}(X)$, respectively.

Since $X$ generates the free group $F_{a}(X)$, each element $g\in F_{a}(X)$ has the form $g=x_{1}^{\varepsilon_{1}}\cdots x_{n}^{\varepsilon_{n}}$, where $x_{1}, \cdots, x_{n1}\in X$ and $\varepsilon_{1}, \cdots, \varepsilon_{n}=\pm 1$. This word for $g$ is called {\it reduced} if it contains no pair of consecutive symbols of the form $xx^{-1}$ or $x^{-1}x$. It follow that if the word $g$ is reduced and non-empty, then it is different from the neutral element of $F_{a}(X)$. In particular, each element $g\in F_{a}(X)$ distinct from the neutral element can be uniquely written in the form $g=x_{1}^{r_{1}}x_{2}^{r_{2}}\cdots x_{n}^{r_{n}}$, where $n\geq 1$, $r_{i}\in \mathbb{Z}\setminus\{0\}$, $x_{i}\in X$, and $x_{i}\neq x_{i+1}$ for each $i=1, \cdots, n-1$. Such a word is called the {\it normal form} of $g$. Similar assertions are valid for $A_{a}(X)$.

We denote by $\mathbb{N}$ the set of all natural
numbers. The letter $e$
denotes the neutral element of a group. Readers may consult
\cite{A2008, E1989, Gr1984} for notations and terminology not
explicitly given here.

\bigskip

\section{Backgrounds}
If $X$ is an arbitrary subspace of a Tychonoff space $Y$, then let $e_{X, Y}$ be the natural embedding mapping from $X$ to $Y$. The following two theorems are well known in the theory of free topological groups.

\begin{theorem}\cite[Nummela-Pestov]{NE, PV}\label{t1}
Let $X$ be a dense subspace of a Tychonoff space $Y$. Then the embedding mapping $e_{X, Y}$ can be extended to a topological monomorphism $\hat{e}_{X, Y}: F(X)\rightarrow F(Y)$ if and only if $X$ is P-embedded in $Y$.
\end{theorem}

\begin{theorem}\cite[M.G. Tkackenko]{TM1}\label{t2}
Let $X$ be an arbitrary subspace of a Tychonoff space $Y$. Then the embedding mapping $e_{X, Y}$ can be extended to a topological monomorphism $\hat{e}_{X, Y}: A(X)\rightarrow A(Y)$ if and only if $X$ is P$^{\ast}$-embedded in $Y$.
\end{theorem}

Obviously, if $X$ is a subspace of a Tychonoff space $Y$, then the embedding mapping $e_{X, Y}: X\rightarrow Y$ can be extended to a continuous monomorphism $\hat{e}_{X, Y}: PG(X)\rightarrow PG(Y)$. However, by Theorems~\ref{t1} and~\ref{t2}, it is natural to ask the following two questions:

\begin{question}\label{q1}
Let $X$ be a dense subspace of a Tychonoff space $Y$. Is it true that the embedding mapping $e_{X, Y}$ can be extended to a topological monomorphism $\hat{e}_{X, Y}: FP(X)\rightarrow FP(Y)$ if and only if $X$ is quasi-P-embedded in $Y$?
\end{question}

\begin{question}\label{q2}
Let $X$ be a subspace of a Tychonoff space $Y$. Is it true that the embedding mapping $e_{X, Y}$ can be extended to a topological monomorphism $\hat{e}_{X, Y}: AP(X)\rightarrow AP(Y)$ if and only if $X$ is quasi-P$^{\ast}$-embedded in $Y$?
\end{question}

In this paper, we shall give an affirmative answer to Question~\ref{q2}. Moreover, we shall give a partial answer to Question~\ref{q1}, and prove that for a Tychonoff space $Y$ if $X$ is a dense subspace of the smallest uniformity containing $\mathscr{U}_{Y}$ induces on $\tilde{Y}$ of the bicompletion of $(Y, \mathscr{U}_{Y})$ and the natural mapping $\hat{e}_{X, Y}: FP(X)\rightarrow FP(Y)$ is a topological monomorphism then $X$ is quasi-P-embedded in $Y$.

\bigskip

\section{Quasi-pseudometrics on free paratopological groups}
In this section, we shall give some lemmas and theorems in order to prove our
main results in Section 4.

We now outline some of the ideas of \cite{RS} in a form suitable for our applications.

Suppose that $e$ is the neutral element of the abstract free group $F_{a}(X)$ on $X$, and suppose that $\rho$ is a fixed quasi-pseduometric on $X$ which is bounded by 1. Extend $\rho$ from $X$ to a quasi-pseudometric $\rho_{e}$ on $X\cup\{e\}$ by putting
\[\rho_{e}(x, y)=\left\{
\begin{array}{lll}
0, & \mbox{if } x=y,\\
\rho(x, y), & \mbox{if } x, y\in X,\\
1, & \mbox{otherwise}\end{array}\right.\]
for arbitrary $x, y\in X\cup\{e\}$. By \cite{RS}, we extend $\rho_{e}$ to a quasi-pseudometric $\rho^{\ast}$ on $\tilde{X}=X\cup\{e\}\cup X^{-1}$ defined by
\[\rho^{\ast}(x, y)=\left\{
\begin{array}{lll}
0, & \mbox{if } x=y,\\
\rho_{e}(x, y), & \mbox{if } x, y\in X\cup\{e\},\\
\rho_{e}(y^{-1}, x^{-1}), & \mbox{if } x, y\in X^{-1}\cup\{e\},\\
2, & \mbox{otherwise}\end{array}\right.\] for arbitrary $x, y\in\tilde{X}$.

Let $A$ be a subset of $\mathbb{N}$ such that $|A|=2n$ for some $n\geq 1$. A {\it scheme} on $A$ is a partition of $A$ to pairs $\{a_{i}, b_{i}\}$ with $a_{i}<b_{i}$ such that each two intervals $[a_{i}, b_{i}]$ and $[a_{j}, b_{j}]$ in $\mathbb{N}$ are either disjoint or one contains the other.

If $\mathcal{X}$ is a word in the alphabet $\tilde{X}$, then we denote the reduced form and the length of  $\mathcal{X}$ by $[\mathcal{X}]$ and $\ell (\tilde{X})$ respectively.

For each $n\in \mathbb{N}$, let $\mathcal{S}_{n}$ be the family of all schemes $\varphi$ on $\{1, 2, \cdots, 2n\}$. As in \cite{RS}, define
$$\Gamma_{\rho}(\mathcal{X}, \varphi)=\frac{1}{2}\sum_{i=1}^{2n}\rho^{\ast}(x_{i}^{-1}, x_{\varphi (i)}).$$
Then we define a quasi-prenorm $N_{\rho}: F_{a}(X)\rightarrow [0, +\infty)$ by setting $N_{\rho}(g)=0$ if $g=e$ and $$N_{\rho}(g)=\inf\{\Gamma_{\rho}(\mathcal{X}, \varphi): [\mathcal{X}]=g, \ell (\tilde{X})=2n, \varphi\in\mathcal{S}_{n}, n\in \mathbb{N}\}$$ if $g\in F_{a}(X)\setminus\{e\}$. It follows from Claim 3 in \cite{RS} that $N_{\rho}$ is an invariant quasi-prenorm on $F_{a}(X)$. Put $\hat{\rho}(g, h)=N_{\rho}(g^{-1}h)$ for all $g, h\in F_{a}(X)$. We refer to $\hat{\rho}$ as the Graev extension of $\rho$ to $F_{a}(X)$.

Given a word $\mathcal{X}$ in the alphabet $\tilde{X}$, we say that $\mathcal{X}$ is {\it almost irreducible} if $\tilde{X}$ does not contain two consecutive symbols of the form $u$, $u^{-1}$ or $u^{-1}$, $u$ (but $\mathcal{X}$ may contain several letters equal to $e$), see \cite{RS}.

The following two lemmas are essentially Claims in the proof of Theorem 3.2 in \cite{RS}.

\begin{lemma}\label{l10}\cite{RS}
Let $\varrho$ be a quasi-pseudometric on $X$ bounded by 1. If $g$ is a reduced word in $F_{a}(X)$ distinct from $e$, then there exists an almost irreducible word $\mathcal{X}_{g}=x_{1}x_{2}\cdots x_{2n}$ of length $2n\geq 2$ in the alphabet $\tilde{X}$ and a scheme $\varphi_{g}\in\mathcal{S}_{n}$ that satisfy the following conditions:
\begin{enumerate}
\item for $i=1, 2, \cdots, 2n$, either $x_{i}$ is $e$ or $x_{i}$ is a letter in $g$;

\item $[\mathcal{X}_{g}]=g$ and $n\leq \ell(g)$; and

\item $N_{\rho}(g)=\Gamma_{\rho}(\mathcal{X}_{g}, \varphi_{g}).$
\end{enumerate}
\end{lemma}

\begin{lemma}\label{l11}\cite{RS}
The family $\mathcal{N}=\{U_{\rho}(\varepsilon): \varepsilon >0\}$ is a base at the neutral element $e$ for a paratopological group topology $\mathscr{F}_{\rho}$ on $F_{a}(X)$, where $U_{\rho}(\varepsilon)=\{g\in F_{a}(X): N_{\rho}(g)<\varepsilon\}$. The restriction of $\mathscr{F}_{\rho}$ to $X$ coincides with the topology of the space $X$ generated by $\rho$.
\end{lemma}

\begin{lemma}\label{l0}\cite{F1982}
For every sequence $V_{0}, V_{1}, \cdots,$ of elements of a quasi-uniformity $\mathscr{U}$ on a set $X$, if $$V_{0}=X\times X\ \mbox{and}\ V_{i+1}\circ V_{i+1}\circ V_{i+1}\subset V_{i},\ \mbox{for}\ i\in \mathbb{N},$$ where `$\circ$' denotes the composition of entourages in the quasi-uniform space $(X, \mathscr{U})$, then there exists a quasi-pseudometric $\rho$ on the set $X$ such that, for each $i\in \mathbb{N}$, $$V_{i}\subset\{(x, y): \rho (x, y)\leq \frac{1}{2^{i}}\}\subset V_{i-1}.$$
\end{lemma}

\begin{lemma}\label{l1}
For every quasi-uniformity $\mathscr{V}$ on a set $X$ and each $V\in \mathscr{V}$ there exists a quasi-pseudometric $\rho$ bounded by 1 on $X$ which is quasi-uniform with respect to $\mathscr{V}$ and satisfies the condition $$\{(x, y): \rho (x, y)< 1\}\subset V.$$
\end{lemma}

\begin{proof}
By the definition of
a quasi-uniformity, we can find a sequence $V_{0}, V_{1}, \cdots , V_{n}, \cdots$ of members of $\mathscr{V}$ such that $V_{1}=V$ and $V_{i+1}\circ V_{i+1}\circ V_{i+1}\subset V_{i},\ \mbox{for each}\ i\in \mathbb{N}$. Let $\rho=\mbox{min}\{1, 4\rho_{0}\}$, where $\rho_{0}$ is a quasi-pseudometric as in Lemma~\ref{l0}. Then $\rho$ is a quasi-pseudometric which has the required property.
\end{proof}

Given a finite subset $B$ of $\mathbb{N}$ on $B$ with $|B|=2n\geq 2$, we say that a bijection $\varphi: B\rightarrow B$ is an {\it Abelian scheme} on $B$ if $\varphi$ is an involution without fixed points, that is, $\varphi(i)=j$ always implies $j\neq i$ and $\varphi(j)=i$.

\begin{lemma}\label{l2}
Suppose that $\rho$ is a quasi-pseudometric on a set $X$, and suppose that $m_{1}x_{1}+\cdots +m_{n}x_{n}$ is the normal form of an element $h\in A_{a}(X)\setminus\{e\}$ of the length $l=\sum_{i=1}^{n}|m_{i}|$. Then there is a representation

$h=(-u_{1}+v_{1})+\cdots +(-u_{k}+v_{k})$ $\dotfill$ (1)\\
where $2k=l$ if $l$ is even and $2k=l+1$ if $l$ is odd, $u_{1}, v_{1}, \cdots , u_{k}, v_{k}\in\{\pm x_{1}, \cdots , \pm x_{n}\}$ (but $v_{k}=e$ if $l$ is odd), and such that

$\hat{\rho}_{A}(e, h)=\sum_{i=1}^{k}\rho^{\ast}(u_{i}, v_{i})$. $\dotfill$ (2)\\

In addition, if $\hat{\rho}_{A}(e, h)<1$, then $l=2k$, and one can choose $y_{1}, z_{1}, \cdots , y_{k}, z_{k}\in\{x_{1}, \cdots , x_{n}\}$ such that

$h=(-y_{1}+z_{1})+\cdots +(-y_{k}+z_{k})$ $\dotfill$ (3)\\ and

$\hat{\rho}_{A}(e, h)=\sum_{i=1}^{k}\rho^{\ast}(y_{i}, z_{i})$. $\dotfill$ (4).
\end{lemma}

\begin{proof}
Obviously, we have $h=h_{1}+\cdots +h_{l}$, where $h_{i}\in\{\pm x_{1}, \cdots , \pm x_{n}\}$ for each $1\leq i\leq l$. Obviously, there exists an integer $k$ such that $2k-1\leq l\leq 2k$. Without loss of generality, we may assume that $l$ is even. In fact, if $l=2k-1$, then one can additionally put $h_{2k}=e$. It follows from the proof of Lemma~\ref{l10} (see \cite{RS}) that we have a similar assertion is valid for the case of $A_{a}(X)$. Then there exists an Abelian scheme $\varphi$ on $\{1, 2, \cdots , 2k\}$ such that

$\hat{\rho}_{A}(e, h)=\frac{1}{2}\sum_{i=1}^{2k}\rho^{\ast}(-h_{i}, h_{\varphi (i)}).$ $\dotfill$ (5)\\

Since the group $A_{a}(X)$ is Abelian, we may assume that $\varphi (2i-1)=2i$ for each $1\leq i\leq k$. Obviously, $\varphi (2i)=2i-1$ for each $1\leq i\leq k$. Hence, we have

$h=(h_{1}+h_{2})\cdots +(h_{2k-1}+h_{2k})$. $\dotfill$ (6)\\

For each $1\leq i\leq k$, put $u_{i}=-h_{2i-1}$ and $v_{i}=h_{2i}$. Then it follows from (5) and (6) that (1) and (2) are true.

Finally, suppose that $\hat{\rho}_{A}(e, h)<1$. Since $\rho (x, e)= 1$ and $\rho (e, x)= 1$, we have $\rho^{\ast}(x, e)= 1$, $\rho^{\ast}(e, x)= 1$, $\rho^{\ast}(-x, y)= 2$ and $\rho^{\ast}(x, -y)= 2$ for all $x, y\in X$. However, it follows from (5) that $\rho^{\ast}(-h_{2i-1}, h_{2i})<1$ for each $1\leq i\leq k$, and
therefore, one of the elements $h_{2i-1}, h_{2i}$ in $X$ while the other is in $-X$. Thus, for each $1\leq i\leq k$, we have $h_{2i-1}+h_{2i}=-y_{i}+z_{i}$, where $y_{i}, z_{i}\in X$. Obviously, $y_{i}, z_{i}\in\{x_{1}, \cdots , x_{n}\}$ for each $1\leq i\leq k$. Next, we only need to replace $h_{2i-1}$ and $h_{2i}$ by the corresponding elements $\pm y_{i}$ and $\pm z_{i}$ in (5) and (6), respectively. Hence we obtain (3) and (4).
\end{proof}

\begin{lemma}\label{l8}
If $d$ is a quasi-pseudometric on a set $X$
quasi-uniform  such that it is quasi-uniform with respect to $\mathscr{U}_{X}$, then $d$ is continuous as a mapping from $(X\times X, \mathscr{U}_{X}^{-1}\times \mathscr{U}_{X})$ to $(\mathbb{R}, \mathscr{U}^{\star})$.
\end{lemma}

\begin{proof}
Take an arbitrary point $(x_{0}, y_{0})\in X\times X$. It is sufficient to show that $d$ is continuous at the point $(x_{0}, y_{0})$. For each $\varepsilon>0$, since $d$ is quasi-uniform with respect to $\mathscr{U}_{X}$, there exists an $U\in\mathscr{U}_{X}$ such that $d(x, y)<\frac{\varepsilon}{2}$ for each $(x, y)\in U$. Let $U_{1}=\{x\in X: (x, x_{0})\in U\}$ and $U_{2}=\{y\in X: d(y_{0}, y)<\frac{\varepsilon}{2}\}$. Then $U_{1}, U_{2}$ are neighborhoods of the points $x_{0}$ and $y_{0}$ in the spaces $(X, \mathscr{U}_{X}^{-1})$ and $(X, \mathscr{U}_{X})$ respectively. Put $V=U_{1}\times U_{2}$. Then $V$ is a neighborhood of the point $(x_{0}, y_{0})$ in $(X\times X, \mathscr{U}_{X}^{-1}\times \mathscr{U}_{X})$. For each $(x, y)\in V$, we have
\begin{eqnarray}
d(x, y)-d(x_{0}, y_{0})&\leq&d(x, x_{0})+d(x_{0}, y_{0})+d(y_{0}, y)-d(x_{0}, y_{0}) \nonumber\\
&=&d(x, x_{0})+d(y_{0}, y)<\frac{\varepsilon}{2}+\frac{\varepsilon}{2}=\varepsilon. \nonumber
\end{eqnarray}
Therefore, the quasi-pseudometric $d$ is continuous at the point $(x_{0}, y_{0})$.
\end{proof}

\begin{lemma}\label{l3}\cite{NP}
Let $\{V_{i}: i\in\mathbb{N}\}$ be a sequence of subsets of a group $G$ with identity $e$ such that $e\in V_{i}$ and $V_{i+1}^{3}\subset V_{i}$ for each $i\in \mathbb{N}$. If $k_{1}, \cdots , k_{n}, r\in \mathbb{N}$ and $\sum_{i=1}^{n}2^{-k_{i}}\leq 2^{-r}$, then we have $V_{k_{1}}\cdots V_{k_{n}}\subset V_{r}$.
\end{lemma}

In the next theorem we prove that the family of quasi-pseudometrics $\{\hat{\rho}_{A}: \rho\in \mathscr{P}_{X}\}$, where $\mathscr{P}_{X}$ is the family of all continuous quasi-pseudometrics from $(X\times X, \mathscr{U}_{X}^{-1}\times \mathscr{U}_{X})$ to $(\mathbb{R}, \mathscr{U}^{\star})$, generates the topology of the free Abelian paratopological group $AP(X)$.

\begin{theorem}\label{t0}
Let $X$ be a Tychonoff space, and
let $\mathscr{P}_{X}$ be the family of all continuous quasi-pseudometrics from $(X\times X, \mathscr{U}_{X}^{-1}\times \mathscr{U}_{X})$ to $(\mathbb{R}, \mathscr{U}^{\star})$ which are bounded by 1. Then the sets $$V_{\rho}=\{g\in AP(X): \hat{\rho}_{A}(e, g)<1\}$$ with $\rho\in\mathscr{P}_{X}$ form a local base at the neutral element $e$ of $AP(X)$.
\end{theorem}

\begin{proof}
Let $V$ be an open neighborhood of $e$ in $AP(X)$. Since $AP(X)$ is a paratopological group, there exists a sequence $\{V_{n}: n\in \mathbb{N}\}$ of open neighborhoods of $e$ in $AP(X)$ such that $V_{1}\subset V$ and $V_{i+1}+V_{i+1}+V_{i+1}\subset V_{i}$ for every $i\in \mathbb{N}$. For each $n\in \mathbb{N}$, put $$U_{n}=\{(x, y)\in X\times X: -x+y\in V_{n}\}.$$ Then each $U_{n}$ is an element of the universal quasi-uniformity $\mathscr{U}_{X}$ on the space $X$ and $U_{n+1}\circ U_{n+1}\circ U_{n+1}\subset U_{n}$. Hence, it follows from Lemmas~\ref{l0} and~\ref{l8} that there is a continuous quasi-pseudometric $\rho_{1}$ on $X$ such that, for each $n\in \mathbb{N},$$$\{(x, y)\in X\times X: \rho_{1} (x, y)<2^{-n}\}\subset U_{n}.$$Let $\rho =\mbox{min}\{1, 4\rho_{1}\}$. Then $\rho\in\mathscr{P}_{X}$.

Claim: We have $V_{\rho}\subset V.$

Indeed, let $h\in V_{\rho}$. It follows from Lemma~\ref{l2} that the element $h$ can be written in the form $$h=(-x_{1}+y_{1})+\cdots +(-x_{m}+y_{m}),\ \mbox{where}\ x_{i}, y_{i}\in X\ \mbox{for each}\ 1\leq i\leq m,$$ such that $$\hat{\rho}_{A}(e, h)=\rho (x_{1}, y_{1})+\cdots +\rho (x_{m}, y_{m})< 1.$$ It follows from the definition of $\rho$ and $\rho_{1}$ that $\hat{\rho}=4\hat{\rho_{1}}$. Therefore, we have $$\hat{\rho}_{1}(e, h)=\rho_{1}(x_{1}, y_{1})+\cdots +\rho_{1}(x_{m}, y_{m})<\frac{1}{4}.$$
If $1\leq i\leq m$ and $\rho_{1}(x_{i}, y_{i})>0$ then we choose a $k_{i}\in \mathbb{N}$ such that $$2^{-k_{i}-1}\leq \rho_{1}(x_{i}, y_{i})<2^{-k_{i}}.$$ And then, if $1\leq i\leq m$ and $\rho_{1}(x_{i}, y_{i})=0$ then we choose a sufficiently large $k_{i}\in \mathbb{N}$ such that $\sum_{i=1}^{m}2^{-k_{i}}<\frac{1}{2}$. For every $1\leq i\leq m$, since $-x_{i}+y_{i}\in V_{k_{i}}$, it follows from Lemma~\ref{l3} that $$h=(-x_{1}+y_{1})+\cdots +(-x_{m}+y_{m})\in V_{k_{1}}+\cdots +V_{k_{m}}\subset V_{1}\subset V.$$ Therefore, we have $V_{\rho}\subset V.$
\end{proof}

We don't know whether a similar assertion is valid for the free paratopological group $FP(X)$.

However, we have the following Theorem~\ref{t3}. Our argument will be based on the following combinatorial lemma (Readers can consult the proof in \cite[Lemma 7.2.8]{A2008}.).

\begin{lemma}\label{l4}
Let $g=x_{1}\cdots x_{2n}$ be a reduced element of $F_{a}(X)$, where $x_{1}, \cdots , x_{2n}\in X\cup X^{-1}$, and let $\varphi$ be a scheme on $\{1, 2, \cdots, 2n\}$. Then there are natural numbers $1\leq i_{1}<\cdots <i_{n}\leq 2n$ and elements $h_{1}, \cdots , h_{n}\in F_{a}(X)$ satisfying the following two conditions:\\
i) $\{i_{1}, \cdots , i_{n}\}\cup\{i_{\varphi(1)}, \cdots , i_{\varphi(n)}\}=\{1, 2, \cdots , 2n\};$\\
ii) $g=(h_{1}x_{i_{1}}x_{\varphi(i_{1})}h_{1}^{-1})\cdots (h_{n}x_{i_{n}}x_{\varphi(i_{n})}h_{n}^{-1}).$
\end{lemma}

A paratopological group $G$ has an {\it invariant
basis}  if there exists a family $\mathscr{L}$ of continuous and invariant quasi-pseudometric on $G$ such that the family $\{U_{\rho}: \rho\in\mathscr{L}\}$ as a base at the neutral element $e$ in $G$, where each $U_{\rho}=\{g\in G: \rho(e, g)<1\}$.

\begin{theorem}\label{t3}
For each Tychonoff space $X$, if the abstract group $F_{a}(X)$ admits the maximal paratopological group topology $\mathscr{F}_{\mbox{inv}}$ with invariant basis such that every continuous mapping $f: X\rightarrow H$ to a paratopological group $H$ with invariant basis can be extended to a continuous homomorphism $\tilde{f}: (F_{a}(X), \mathscr{F}_{\mbox{inv}})\rightarrow H$, then the family of all sets of the form $$U_{\rho}=\{g\in F_{a}(X): \hat{\rho}(e, g)<1\},$$where $\rho$ is a continuous quasi-pseudometric from $(X\times X, \mathscr{U}_{X}^{-1}\times \mathscr{U}_{X})$ to $(\mathbb{R}, \mathscr{U}^{\star})$, with $\rho\leq 1$ constitutes a base of the topology $\mathscr{F}_{\mbox{inv}}$ at the neutral element $e$ of $F_{a}(X)$.
\end{theorem}

\begin{proof}
For each $\rho\in\mathscr{P}_{X}$, put $$U_{\rho}=\{g\in F_{a}(X): N_{\rho}(g)<1\},\ \mbox{and}\ \mathscr{N}=\{U_{\rho}: \rho\in\mathscr{P}_{X}\}, $$where $N_{\rho}(g)$ is the invariant quasi-prenorm on $F_{a}(X)$ defined by $\hat{\rho}(g, h)=N_{\rho}(g^{-1}h)$. By Lemma~\ref{l11} and Proposition 3.8 in \cite{PN}, it is easy to see that $\mathscr{N}$ as a base at the neutral element $e$ of $F_{a}(X)$ for a Hausdorff paratopological group topology. We denote this topology by $\mathscr{F}_{\mbox{inv}}$. Since $\hat{\rho}$ is invariant on $F_{a}(X)$, the paratopological group $FP_{\mbox{inv}}(X)=(FP(X), \mathscr{F}_{\mbox{inv}})$ has an invariant basis, and hence $\hat{\rho}$ is continuous on $FP_{\mbox{inv}}(X)$.

Let $f: X\rightarrow H$ be a continuous mapping of $X$ to a paratopological group $H$ with invariant basis. Let $\tilde{f}$ be the extension of $f$ to a homomorphism of $F_{a}(X)$ to $H$.

Claim: The map $\tilde{f}: FP_{\mbox{inv}}(X)\rightarrow H$ is a continuous homomorphism.

Let $V$ be an open neighborhood of the neutral element of $H$. Then there exists an invariant quasi-prenorm $N$ on $H$ such that $W=\{h\in H: N(h)<1\}\subset V$ by Lemma~\ref{l1}. Therefore, we can define a quasi-pseudometric $\rho$ on $X$ by $\rho (x, y)=N(f^{-1}(x)f(y))$ for all $x, y\in X$. Next, we shall show that $\tilde{f}(U_{\rho})\subset W$. Indeed, take an arbitrary reduced element $g\in U_{\rho}$ distinct from the neutral element of $F_{a}(X)$. Obviously, we have $\hat{\rho}(e, g)<1$. Moreover, it is easy to see that $g$ has even length, say $g=x_{1}\cdots x_{2n}$, where $x_{i}\in X\cup X^{-1}$ for each $1\leq i\leq 2n$. It follows from $\hat{\rho}(e, g)<1$ that there is a scheme $\varphi$ on $\{1, 2, \cdots, 2n\}$ such that $$\hat{\rho}(e, g)=\frac{1}{2}\sum_{i=1}^{2n}\rho^{\ast}(x_{i}^{-1}, x_{\varphi(i)})<1.$$By Lemma~\ref{l4}, we can find a partition $\{1, 2, \cdots , 2n\}=\{i_{1}, \cdots , i_{n}\}\cup\{i_{\varphi(1)}, \cdots , i_{\varphi(n)}\}$ and a representation of $g$ as a product $g=g_{1}\cdots g_{n}$ such that $g_{k}=h_{k}x_{i_{k}}x_{\varphi(i_{k})}h_{k}^{-1}$ for each $k\leq n$, where $h_{k}\in F_{a}(X)$. Since $N$ is invariant, we have
\begin{eqnarray}
N(\tilde{f}(g))&\leq&\sum_{i=1}^{n}N(\tilde{f}(g_{k}))=\sum_{i=1}^{n}N(\tilde{f}(x_{k})\tilde{f}(x_{\varphi(k)})) \nonumber\\
&=&\rho^{\ast}(x_{1}^{-1}, x_{\varphi(1)})+\cdots +\rho^{\ast}(x_{n}^{-1}, x_{\varphi(n)})\nonumber\\
&<&1.\nonumber
\end{eqnarray}
Therefore, we have $\tilde{f}(g)\in W$, and it follows that $\tilde{f}(U_{\rho})\subset W\subset V$. Hence $\tilde{f}$ is a continuous homomorphism.
\end{proof}

\bigskip

\section{Topological monomorphisms between free paratopological groups}
In order to prove one of our main theorems, we also need the following lemma.

\begin{lemma}\label{l7}
Let $(X, \mathscr{U}_{X})$ be a quasi-uniform subspace of a Tychonoff space $(Y, \mathscr{U}_{Y})$. Then $X$ is quasi-P$^{\ast}$-embedded in $Y$.
\end{lemma}

\begin{proof}
Let $d$ be a bounded, continuous quasi-pseudometric from $(X\times X, \mathscr{U}_{X}^{-1}\times \mathscr{U}_{X})$ to $(\mathbb{R}, \mathscr{U}^{\star})$. One can assume that $d$ is bounded by $\frac{1}{2}$. For each $i\in \mathbb{N}$, take a $V_{i}\in \mathscr{U}_{Y}$ satisfying $V_{i}\cap (X\times X)\subset \{(x, y)\in X\times X: d(x, y)<\frac{1}{2^{i}}\}$, and then by \cite[Chap. 3, Proposition 2.4 and Theorem 2.5]{SA}, take a continuous quasi-pseudometric $d_{i}$ from $(Y\times Y, \mathscr{U}_{Y}^{-1}\times \mathscr{U}_{Y})$ to $(\mathbb{R}, \mathscr{U}^{\star})$ such that $d_{i}$ is bounded by 1, quasi-uniform with respect to $\mathscr{U}_{Y}$ and $\{(x, y)\in Y\times Y: d_{i}(x, y)<\frac{1}{4}\}\subset V_{i}$. Put $$\rho(x, y) =8\sum_{i=1}^{\infty}\frac{1}{2^{i}}d_{i}(x, y).$$ One can easily prove that $\rho$ is a continuous quasi-pseudometric from $(Y\times Y, \mathscr{U}_{Y}^{-1}\times \mathscr{U}_{Y})$ to $(\mathbb{R}, \mathscr{U}^{\star})$. Moreover, it is easy to see that $\rho$ is quasi-uniform with respect to $\mathscr{U}_{Y}$ and satisfies $d(x, y)\leq\rho(x, y)$ for all $x, y\in X$. Put $$\rho^{\prime}(x, y)=\inf\{\rho(x, a)+d(a, b)+\rho(b, y): a, b\in X\},\ \mbox{where}\ x, y\in Y.$$ Let $$\tilde{d}=\min\{\rho(x, y), \rho^{\prime}(x, y)\}.$$ Obviously, $\tilde{d}$ is quasi-uniform with respect to $\mathscr{U}_{Y}$. It follows from Lemma~\ref{l8} that $\tilde{d}$ is a continuous quasi-pseudometric from $(Y\times Y, \mathscr{U}_{Y}^{-1}\times \mathscr{U}_{Y})$ to $(\mathbb{R}, \mathscr{U}^{\star})$. Moreover, we have $\tilde{d}|X\times X=d.$ Therefore, $X$ is quasi-P$^{\ast}$-embedded in $Y$.
\end{proof}

Now, we shall prove our main theorem, which gives an affirmative answer to Question~\ref{q2}.

\begin{theorem}
Let $X$ be an arbitrary subspace of a Tychonoff space $Y$. Then the natural mapping $\hat{e}_{X, Y}: AP(X)\rightarrow AP(Y)$ is a topological monomorphism if and only if $X$ is quasi-P$^{\ast}$-embedded in $Y$.
\end{theorem}

\begin{proof}
Necessity. Let $d$ be an arbitrary bounded continuous quasi-pseudometric from $(X\times X, \mathscr{U}_{X}^{-1}\times \mathscr{U}_{X})$ to $(\mathbb{R}, \mathscr{U}^{\star})$, where $\mathscr{U}_{X}$ is the universal quasi-uniformity on $X$. Then $U_{d}=\{(x, y)\in X\times X: d(x, y)<1\}\in\mathscr{U}_{X}$. Put $V_{d}=\{g\in AP(X): \hat{d}(e, g)<1\}$. Then $V_{d}$ is a neighborhood of the neutral element of $AP(X)$. Since $AP(X)\subset AP(Y)$, it follows from Theorem~\ref{t0} that there is some continuous quasi-pseudometric $\rho$ from $(Y\times Y, \mathscr{U}_{Y}^{-1}\times \mathscr{U}_{Y})$ to $(\mathbb{R}, \mathscr{U}^{\star})$ such that $V_{\rho}\cap AP(X)\subset V_{d}$, where $\mathscr{U}_{Y}$ is the universal quasi-uniformity on $Y$ and $V_{\rho}=\{g\in AP(Y): \hat{\rho}(e, g)<1\}$. Note that $U_{\rho}=\{(x, y)\in Y\times Y: \rho(x, y)<1\}\in\mathscr{U}_{Y}$ and $U_{\rho}\cap (X\times X)\subset U_{d}$. Moreover, one can see that $\hat{\rho}(e, x^{-1}y)=\rho(x, y)$ and $\hat{d}(e, x^{-1}y)=d(x, y)$ for all $x, y$. Therefore, $(X, \mathscr{U}_{X})$ is a quasi-uniform subspace of $(Y, \mathscr{U}_{Y})$. Hence $X$ is quasi-P$^{\ast}$-embedded in $Y$ by Lemma~\ref{l7}.

Sufficiency. Let $X$ be quasi-P$^{\ast}$-embedded in $Y$. Denote by $e_{X, Y}$ the identity embedding of $X$ in $Y$. Obviously, the monomorphism $\hat{e}_{X, Y}$ is continuous. Next, we need to show that the isomorphism $\hat{e}_{X, Y}^{-1}: AP(X, Y)\rightarrow AP(X)$ is continuous. Assume that $U$ is a neighborhood of the neutral element $e_{X}$ in $AP(X)$. It follows from Theorem~\ref{t0} that there is a continuous quasi-pseudometric $\rho$ from $(X\times X, \mathscr{U}_{X}^{-1}\times \mathscr{U}_{X})$ to $(\mathbb{R}, \mathscr{U}^{\star})$ such that $V_{\rho}=\{g\in AP(X): \hat{\rho}_{A}(e_{X}, g)<1\}\subset U$. Without loss of generality, we may assume that $\rho\leq 1$ (otherwise, replace $\rho$ with $\rho^{\prime}=\mbox{min}\{\rho, 1\}$). Since $X$ is quasi-P$^{\ast}$-embedded in $Y$, the quasi-pseudometric $\rho$ can be extended to a continuous quasi-pseudometric $d$ from $(Y\times Y, \mathscr{U}_{Y}^{-1}\times \mathscr{U}_{Y})$ to $(\mathbb{R}, \mathscr{U}^{\star})$. Suppose that $\hat{d}_{A}$ is the Graev extension of $d$ over $AP(Y)$. It follows from Theorem~\ref{t0} again that $V_{d}=\{g\in AP(Y): \hat{d}_{A}(e_{Y}, g)<1\}$ is an open neighborhood of the neutral element $e_{Y}$ in $AP(Y)$. Obviously, one can identify the abstract group $A_{a}(X)$ with the subgroup $\hat{e}_{X, Y}(A_{a}(X))=A_{a}(X, Y)$ of $A_{a}(Y)$ generated by the subset $X$ of $A_{a}(Y)$. Since $d\mid X=\rho$, it follows from Lemma~\ref{l2} that, for each $h\in A_{a}(X, Y)$, $\hat{d}_{A}(e_{Y}, h)=\hat{\rho}_{A}(e_{Y}, h)$. Hence we have $A_{a}(X, Y)\cap V_{d}=V_{\rho}$, that is, $AP(X, Y)\cap V_{d}=\hat{e}_{X, Y}(V_{\rho})$. Therefore, the isomorphism $\hat{e}_{X, Y}^{-1}: AP(X, Y)\rightarrow AP(X)$ is continuous.
\end{proof}

In order to give a partial answer to Question~\ref{q1}, we need to prove some lemmas.

\begin{lemma}\label{l6}Let $X$ be a Tychonoff space.
Then the restriction $\mathscr{G}_{X}=\mathscr{G}_{PG(X)}\mid X$ of the left uniformity $\mathscr{G}_{PG(X)}$ of the paratopological group $PG(X)$ to the subspace $X\subset PG(X)$ coincides with the universal quasi-uniformity $\mathscr{U}_{X}$ of $X$.
\end{lemma}

\begin{proof}
Since the topology on $X$ generated by the left uniformity $\mathscr{G}_{PG(X)}$ of $PG(X)$ coincides with the original topology of the space $X$, we have $\mathscr{G}_{X}\subset \mathscr{U}_{X}$. Next, we need to show that $\mathscr{U}_{X}\subset\mathscr{G}_{X}$. Take an arbitrary element $U\in \mathscr{U}_{X}$. It follows from Lemmas~\ref{l1} and~\ref{l8} that there exists a continuous quasi-pseudometric $\rho$ from $(X\times X, \mathscr{U}_{X}^{-1}\times \mathscr{U}_{X})$ to $(\mathbb{R}, \mathscr{U}^{\star})$ such that $\{(x, y)\in X\times X: \rho(x, y)<1\}\subset U.$ By Theorem 3.2 in \cite{RS}, the quasi-pseudometric $\rho$ on set $X$ extends to a left invariant quasi-pseudometric $\hat{\rho}$ on the abstract group $PG(X)$. One can see that $\hat{\rho}$ is continuous from $(PG(X)\times PG(X), \mathscr{U}_{PG(X)}\times \mathscr{U}_{PG(X)}^{-1})$ to $(\mathbb{R}, \mathscr{U}^{\star})$. It follows from Theorem~\ref{t0} that $V=\{g\in PG(X): \hat{\rho}(e, g)<1\}$ is an open neighborhood of the neutral element $e$ in $PG(X)$. If $x, y\in X$ and $x^{-1}y\in V$, then $$\rho(x, y)=\hat{\rho}(x, y)=\hat{\rho}(e, x^{-1}y)<1,$$which implies that the element $W_{V}^{l}=\{(g, h)\in G\times G: g^{-1}h\in V\}$ of $\mathscr{G}_{PG(X)}$ satisfies $W_{V}^{l}\cap (X\times X)\subset U$. Therefore, $\mathscr{U}_{X}\subset\mathscr{G}_{X}$.
\end{proof}

\begin{lemma}\label{l5}\cite{RS1}
The finest quasi-uniformity of each quasi-pseudometrizable topological space is bicomplete.
\end{lemma}

\begin{lemma}\label{l9}
Let $X$ be a subspace of a Tychonoff space $Y$, and let $X$ be $\tau(\tilde{\mathscr{U}_{Y}}^{\ast})$-dense in $(\tilde{Y}, \tilde{\mathscr{U}_{Y}})$, where $\mathscr{U}_{Y}$ is the universal quasi-uniformity and $(\tilde{Y}, \tilde{\mathscr{U}_{Y}})$ is the bicompletion of $(Y, \mathscr{U}_{Y})$. Then the following conditions are equivalent:
\begin{enumerate}
\item $X$ is quasi-P$^{\ast}$-embedded in $Y$;

\item $X$ is quasi-P-embedded in $Y$;

\item $\mathscr{U}_{Y}\mid X=\mathscr{U}_{X}$;

\item $X\subset Y\subset \tilde{X}$, where $(\tilde{X}, \tilde{\mathscr{U}_{X}})$ is the bicompletion of $(X, \mathscr{U}_{X})$.
\end{enumerate}
\end{lemma}

\begin{proof}
Obviously, $(2)\Rightarrow (1)$. Hence it is suffices to show that $(1)\Rightarrow (3)\Rightarrow (4)\Rightarrow (2).$

$(1)\Rightarrow (3).$ Assume that $X$ is quasi-P$^{\ast}$-embedded in $Y$. For each $U\in \mathscr{U}_{X}$, it follows from Lemmas~\ref{l1} and~\ref{l8} that there exists a bounded continuous quasi-pseudometric $\rho_{X}$ from $(X\times X, \mathscr{U}_{X}^{-1}\times \mathscr{U}_{X})$ to $(\mathbb{R}, \mathscr{U}^{\star})$ such that $$W_{X}=\{(x, x^{\prime})\in X\times X: \rho_{X}(x, x^{\prime})<1\}\subset U.$$ Since $X$ is quasi-P$^{\ast}$-embedded in $Y$, let $\rho_{Y}$ be an extension of $\rho_{X}$ to a continuous quasi-pseudometric from $(Y\times Y, \mathscr{U}_{Y}^{-1}\times \mathscr{U}_{Y})$ to $(\mathbb{R}, \mathscr{U}^{\star})$. Put $$W_{Y}=\{(y, y^{\prime})\in Y\times Y: \rho_{Y}(y, y^{\prime})<1\}.$$ Then it is obvious that $W_{Y}\in \mathscr{U}_{Y}$ and $W_{Y}\cap (X\times X)=W_{X}\subset U.$ Therefore, the quasi-uniformity $\mathscr{U}_{Y}\mid X$ is finer than $\mathscr{U}_{X}$. Moreover, it is clear that $\mathscr{U}_{Y}\mid X\subset \mathscr{U}_{X}$. Hence $\mathscr{U}_{Y}\mid X=\mathscr{U}_{X}$.

$(3)\Rightarrow (4).$ Assume that $\mathscr{U}_{Y}\mid X=\mathscr{U}_{X}$. Let $(\tilde{Y}, \tilde{\mathscr{U}_{Y}})$ be the bicompletion of quasi-uniform space $(Y, \mathscr{U}_{Y})$. Because $\tilde{\mathscr{U}_{Y}}\mid Y=\mathscr{U}_{Y}$, we have $\tilde{\mathscr{U}_{Y}}\mid X=\mathscr{U}_{X}$. Moreover, since $X$ is $\tau(\tilde{\mathscr{U}_{Y}}^{\ast})$-dense in $\tilde{Y}$ and $X\subset Y$, $(\tilde{Y}, \tilde{\mathscr{U}_{Y}})$ is the bicompletion of the quasi-uniform space $(X, \mathscr{U}_{X})$. Hence $X\subset Y\subset \tilde{X}$.

$(4)\Rightarrow (2).$ Assume that $Y\subset\tilde{X}$. Consider an arbitrary continuous quasi-pseudometric $\rho$ from $(X\times X, \mathscr{U}_{X}^{-1}\times \mathscr{U}_{X})$ to $(\mathbb{R}, \mathscr{U}^{\star})$. Let $(\overline{X}, \overline{\rho})$ be the quasi-metric space obtained from $(X, \rho)$ by identifying the points of $X$ lying at zero distance one from another with respect to $\rho$. Let $\pi: X\rightarrow \overline{X}$ be the natural quotient mapping. Obviously, $\rho(x, y)=\overline{\rho}(\pi(x), \pi(y))$ for all $x, y\in X$. Suppose that $\mathscr{U}_{\overline{X}}$ is the universal quasi-uniformity on $\overline{X}$. Then $\pi$ is a quasi-uniformly continuous map from $(X, \mathscr{U}_{X})$ to $(\overline{X}, \mathscr{U}_{\overline{X}})$ by \cite{BG}. Moreover, by Lemma~\ref{l5}, $(\overline{X}, \mathscr{U}_{\overline{X}})$ is bicomplete. Therefore, it follows from Theorem 16 in \cite{LW} that $\pi$ admits a quasi-uniformly continuous extension $\overline{\pi}: (\tilde{X}, \tilde{\mathscr{U}}_{X})\rightarrow (\overline{X}, \mathscr{U}_{\overline{X}}).$ Since $Y\subset \tilde{X}$, we can define a continuous mapping of $d$ from $(Y\times Y, \mathscr{U}_{Y}^{-1}\times \mathscr{U}_{Y})$ to $(\mathbb{R}, \mathscr{U}^{\star})$ by $d(x, y)=\overline{\rho}(\overline{\pi}(x), \overline{\pi}(y))$ for all $x, y\in Y$. Clearly, the restriction of $d$ to $X$ coincides with $\rho$. Hence $X$ is quasi-P-embedded in $Y$.
\end{proof}

\begin{theorem}
Let $X$ be an arbitrary $\tau(\tilde{\mathscr{U}_{Y}}^{\ast})$-dense subspace of a Tychonoff space $Y$. If the natural mapping $\hat{e}_{X, Y}: FP(X)\rightarrow FP(Y)$ is a topological monomorphism, then $X$ is quasi-P-embedded in $Y$.
\end{theorem}

\begin{proof}
Assume that the monomorphism $\hat{e}_{X, Y}: FP(X)\rightarrow FP(Y)$ extending the identity mapping $e_{X, Y}: X\rightarrow Y$ is a topological embedding. Therefore, it is easy to see  that we can identify the group $FP(X)$ with the subgroup $FP(X, Y)$ of $FP(Y)$ generated by the set $X$. We denote by $\mathscr{G}_{X}$ and $\mathscr{G}_{Y}$ the left quasi-uniformities of the groups $FP(X)$ and $FP(Y)$, respectively. Since $FP(X)$ is a subgroup of $FP(Y)$, we obtain that $\mathscr{G}_{Y}\mid FP(X)=\mathscr{G}_{X}$. Moreover, it follows from Lemma~\ref{l6} that $\mathscr{G}_{X}\mid X=\mathscr{U}_{X}$ and $\mathscr{G}_{Y}\mid Y=\mathscr{U}_{Y}$. Hence we have $$\mathscr{G}_{Y}\mid X=\mathscr{G}_{X}\mid X=\mathscr{U}_{X}.$$ Therefore, it follows from Lemma~\ref{l9} that $X$ is quasi-P-embedded in $Y$.
\end{proof}

\begin{question}
Let $X$ be an arbitrary $\tau(\tilde{\mathscr{U}_{Y}}^{\ast})$-dense subspace of a Tychonoff space $Y$. If $X$ is quasi-P-embedded in $Y$, is the natural mapping $\hat{e}_{X, Y}: FP(X)\rightarrow FP(Y)$ a topological monomorphism?
\end{question}

In \cite{SO}, O.V. Sipacheva has proved that if $Y$ is a subspace of a Tychonoff $X$ then the subgroup $F(Y, X)$ of $F(X)$ is topologically isomorphic to $F(Y)$ iff $Y$ is $P^{\ast}$-embedded in $X$. In \cite{TM1}, M.G. Tkackenko has proved that if $Y$ is a subspace of a Tychonoff $X$ then the subgroup $A(Y, X)$ of $A(X)$ is topologically isomorphic to $A(Y)$ iff $Y$ is $P^{\ast}$-embedded in $X$.Therefore, we have the following question:

\begin{question}
Let $X$ be an arbitrary subspace of a Tychonoff space $Y$. Is it true that the subgroup $PG(Y, X)$ of $PG(X)$ is topologically isomorphic to $PG(Y)$ iff $Y$ is quasi-$P^{\ast}$-embedded in $X$?
\end{question}

{\bf Acknowledgements}. I wish to thank
the reviewers for the detailed list of corrections, suggestions to the paper, and all her/his efforts
in order to improve the paper.

\end{document}